\documentclass[11pt, reqno, colorlinks]{amsart}

\usepackage{latexsym, amssymb, amsmath, amsthm, amscd}
\usepackage[all,cmtip]{xy}
\usepackage{amsmath}
\usepackage{amsthm}
\usepackage{amssymb}
\usepackage{amsfonts}
\usepackage{amsrefs}
\usepackage{tikz}
\usepackage{amscd} 
\usepackage{comment}
\usepackage{mathrsfs}
\usepackage{comment}
\usepackage[all]{xy}
\usepackage{graphicx}
\usepackage{hyperref}
\usepackage{bbm}
\usepackage{fullpage}
\usepackage[all,cmtip]{xy}
\usepackage{relsize}
\usepackage{amsmath,amscd}
\usepackage{tikz-cd}
\usepackage{tikz}
\usetikzlibrary{matrix}
\usepackage[mathcal]{euscript}
\usepackage{txfonts}

\usepackage{hyperref}

\numberwithin{equation}{section} \DeclareMathSizes{2}{10}{12}{13}
\makeatletter
\newcommand*{\doublerightarrow}[2]{\mathrel{
		\settowidth{\@tempdima}{$\scriptstyle#1$}
		\settowidth{\@tempdimb}{$\scriptstyle#2$}
		\ifdim\@tempdimb>\@tempdima \@tempdima=\@tempdimb\fi
		\mathop{\vcenter{
				\offinterlineskip\ialign{\hbox to\dimexpr\@tempdima+1em{##}\cr
					\rightarrowfill\cr\noalign{\kern.5ex}
					\rightarrowfill\cr}}}\limits^{\!#1}_{\!#2}}}
\newcommand*{\triplerightarrow}[1]{\mathrel{
		\settowidth{\@tempdima}{$\scriptstyle#1$}
		\mathop{\vcenter{
				\offinterlineskip\ialign{\hbox to\dimexpr\@tempdima+1em{##}\cr
					\rightarrowfill\cr\noalign{\kern.5ex}
					\rightarrowfill\cr\noalign{\kern.5ex}
					\rightarrowfill\cr}}}\limits^{\!#1}}}
\makeatother

\newtheorem{thm}{Proposition}[section]
\newtheorem{Thm}[thm]{Theorem}

\newtheorem{lem}[thm]{Lemma}
\newtheorem{defn}[thm]{Definition}

\title{Differential torsion theories on Eilenberg-Moore categories of monads}

\author[ Divya Ahuja]{Divya Ahuja$^{\ast}$}
\author[ Surjeet Kour]{Surjeet Kour}
\address[] {\newline
	Department of Mathematics,
	Indian Institute of Technology Delhi,
	Hauz Khas, New Delhi, 110016, India.}
\email[] {divyaahuja1428@gmail.com, koursurjeet@gmail.com }
\thanks{$^\ast$Corresponding author.}
\keywords{Differential torsion theory, Hereditary torsion theory,  Derivations, Eilenberg-Moore categories}  
\subjclass[2020]{13N15, 16S90, 18C20, 18E40} 

\begin{document}
\begin{abstract}
Let $\mathcal C$ be a Grothendieck category and $U$ be a monad on $\mathcal C$ that is exact and preserves colimits. In this article, we prove that every hereditary torsion theory on the Eilenberg-Moore category of modules over a monad $U$ is differential. Further, if $\delta:U\longrightarrow U$ denotes a derivation on a monad $U$, then we show that every $\delta$-derivation on a $U$-module $M$ extends uniquely to a $\delta$-derivation on the module of quotients of $M$.  
\end{abstract}
\maketitle

	\smallskip
	\hypersetup{linktocpage}
	\section{Introduction}
		Let $R$ be an associative ring with unity and $_{R}Mod$ be the category of unitary left $R$-modules. A map $\delta:R\longrightarrow R$ is said to be a derivation on $R$ if $\delta(a+b) = \delta(a)+\delta(b)$ and $\delta(ab) = \delta(a)b+a\delta(b)$ for all $a,b\in R$. Let $\delta$ be a derivation on $R$ and $M$ be a left $R$-module. Then, a map $d:M\longrightarrow M$ is said to be a $\delta$-derivation on $M$ if $d(m+n) = d(m)+d(n)$ and $d(am) = ad(m)+\delta(a)m$ for all $m, n\in M$ and $a\in R$. Let $\tau=(\mathcal T,\mathcal F)$ be a torsion theory on $_RMod$ and $M_{\tau} \subseteq M$ be the $\tau$-torsion submodule of $M$. Ohatke in \cite{OK} proved that the torsion theory $\tau$ is hereditary if and only if there exists the module of quotients $H_{\tau}(M)$ for every $R$-module $M$. Further, if $\tau$ is hereditary torsion theory, then $H_{\tau}(M)$ can be expressed as the $\tau$-injective envelope of the $\tau$-torsion free module $M/M_{\tau}$ i.e., $H_{\tau}(M) = E_{\tau}(M/M_{\tau})$ (see \cite{PB2}, \cite{Gol}).

		 \smallskip
		 A hereditary torsion theory $\tau$ on $_{R}Mod$ is said to be differential (see \cite{PB}) if for every $M\in$$_{R}Mod$ and every $\delta$-derivation $d$ on $M$, $d(M_{\tau})\subseteq M_{\tau}$. The importance of differential torsion theories lies in their ability to extend every $\delta$-derivation $d$ on $M$ uniquely to a $\delta$-derivation on the module of quotients $H_{\tau}(M)$ of $M$. Golan in \cite{Gol2} proved that any $\delta$-derivation $d$ on $M$ can be extended to a $\delta$-derivation $\bar{d}:H_{\tau}(M)\longrightarrow H_{\tau}(M)$ on  $H_{\tau}(M)$ with respect to any torsion theory $\tau$ on $_RMod$ relative to which $M$ is torsion free. The uniquness of the $\delta$-derivation $\bar{d}$ was established by Bland in \cite[\S2]{PB}. Lomp and van den Berg later proved in \cite{LB} that every hereditary torsion theory on $_RMod$ must be differential.

\smallskip
		In \cite{AB}, Banerjee replaces the associative ring $R$ with a small preadditive category $\mathcal R$ that is equipped with a derivation $\delta$. This follows from Mitchell's idea (see \cite{MB}) that any preadditive category can be seen as a ring with several objects. It has been proved in \cite{AB} that all hereditary torsion theories on the category of right $\mathcal R$-modules must be differential. Further, it is shown that every
		$\delta$-derivation $d$ on an $\mathcal R$-module $\mathscr M$ extends uniquely to a $\delta$-derivation $\bar{d}$ on the module of quotients $H_{\tau}(\mathscr M)$.
		
		\smallskip
		
		In this article, we study the differential torsion theory on the Eilenberg-Moore category of monad. Let $\mathcal C$ be a Grothendieck category i.e., $\mathcal C$ is an abelian category in which colimits exists, filtered colimits are exact, and $\mathcal C$ possesses a set of generators.   
		Consider a monad $(U, \theta, \eta)$ on $\mathcal C$ where $\theta: U\circ U\longrightarrow U$ and $\eta: 1_{\mathcal C}	\longrightarrow U$ are natural transformations satisfying associativity and unit conditions. 
		We start by defining derivation on the monad $U$ as a natural transformation $\delta:U
		\longrightarrow U$ that satisfies the following equality:
	\begin{equation}
		 \theta\circ (1\ast\delta+\delta\ast 1) = \delta\circ \theta
	\end{equation}
	 Let the monad $U$ be exact and preserves colimits. We prove in Proposition \ref{P3.1} that the Eilenberg-Moore category $EM_U$ of modules over the monad $U$ is a Grothendieck category. Further, if $\{G_i\}_{i\in \Lambda}$ is a set of finitely generated projective generators for $\mathcal C$, then $\{UG_i\}_{i\in \Lambda}$ is a set of finitely generated projective generators for $EM_U.$ Further, using \cite[Proposition 3.3]{LLV}, we define a family of Gabriel filters on $EM_U$ and prove that every family of Gabriel filters on $EM_{U}$ is $\delta$-invariant (see, Theorem \ref{T3.4}). 
	 
	 \smallskip
	 Next, we consider a module $(M, f_{M})\in EM_U$. We say that a morphism $D:M\longrightarrow M$ in $\mathcal C$ is a $\delta$-derivation on $M$ if the following equality holds:
	 \begin{equation}
	 	f_{M}\circ (UD+\delta_{M}) = D\circ f_{M}
	 \end{equation}
	 Let $\tau=(\mathcal T, \mathcal F)$ be a hereditary torsion theory on $EM_U$. By \cite[Corollary 3.4]{LLV}, there exists a bijective correspondence between hereditary torsion theories and families of Gabriel filters on $EM_U$. Then, using Theorem \ref{T3.4}, we prove the first main result of our paper which shows that every hereditary torsion theory on $EM_{U}$ is differential (see Theorem \ref{T3.7}). 
	 
	 \smallskip
	 In section 4, we consider the module of quotients $H_{\tau}(M)$ of $M$ with respect to the hereditary torsion theory $\tau$ on $EM_U$. In accordance with the definition given by Banerjee in \cite[Lemma 2.10]{AB}, we define the $\delta$-derivation $\bar{D}:H_{\tau}(M)\longrightarrow H_{\tau}(M)$ on $H_{\tau}(M)$. Finally in Theorem \ref{T4.5}, we prove that for every $\delta$-derivation $D$ on $M$, there exists a unique $\delta$-derivation $\bar{D}$ on $H_{\tau}(M)$ that lifts the $\delta$-derivation $D$.
\section{Preliminaries}
	In this section we recall some definitions and results from \cite{ABKR} and \cite{Mac}. Throughout, we assume that $\mathcal C$ is a Grothendieck category and let $\{G_i\}_{i\in \Lambda}$ be a set of generators for $\mathcal C$.
\subsection{Monad}
A monad $(U,\theta,\eta)$ on a category $\mathcal C$ consists of an endofunctor $U:\mathcal C\longrightarrow \mathcal C$ and two natural transformations $\theta: U\circ U\longrightarrow U$ and $\eta:1_{\mathcal C}\longrightarrow U$ such that for each $M\in \mathcal C$ the following equalities hold:
\begin{equation}\label{eq2.1}
	\theta_{M} \circ \theta_{UM} = \theta_{M}\circ U\theta_{M} \quad \text{and} \quad \theta_{M}\circ U\eta_{M} = 1_{UM} = \theta_{M} \circ \eta_{UM}
\end{equation}
Let $(U,\theta,\eta)$ and $(U',\theta',\eta')$ be two monads on $\mathcal C$. Then, a morphism between $U$ and $U'$ is a natural transformation $\phi: U\longrightarrow U'$ such that for each $M\in\mathcal C,$ we have 
\begin{equation}
	\phi\circ \theta =\theta'\circ (\phi\ast\phi)\quad \text{and} \quad \eta'=\phi\circ \eta\notag
\end{equation}

A module over a monad $(U,\theta,\eta)$ is a pair $(M,f_{M})$, where $M\in\mathcal C$ and $f_{M}:UM\longrightarrow M$ is a morphism in $\mathcal C$ that satisfies
\begin{equation}\label{eq2.2}
	f_M\circ\theta_M= f_M\circ Uf_M\quad \textup{and}\quad f_M\circ\eta_M= 1_M
\end{equation}
Further, a morphism between two $(U,\theta,\eta)$-modules $(M,f_M)$ and $(N,f_{N})$ is a morphism $g:M\longrightarrow N$ in $\mathcal C$ that satisfies
\begin{equation}
	f_{N}\circ Ug= g\circ f_M \notag
\end{equation}
$(U,\theta,\eta)$-modules together with the morphisms as defined above form the Eilenberg-Moore category of monad $(U,\theta,\eta)$, and we denote it by $EM_{U}.$

\smallskip
Given a monad $(U,\theta,\eta)$ and an object $M\in\mathcal C,$ one can observe that $(UM,\theta_{M}:UUM\longrightarrow UM)$ is a $(U,\theta,\eta)$-module. Therefore, there exists a functor $\Phi_{U}:\mathcal C\longrightarrow EM_{U}$ known as the free functor and is defined by setting 
\begin{equation}\label{eq2.3}
	\Phi_{U}(M)=(UM,\theta_{M}) \quad \text{and}\quad \Phi_{U}(f)=Uf
\end{equation}
 for each $M\in\mathcal C$ and $f\in\mathcal C(M,N).$ The functor $\Phi_{U}$ is left adjoint to the forgetful functor $\mathcal F_{U}: EM_{U}\longrightarrow \mathcal C.$ Hence, for each $M\in \mathcal C$ and $N\in EM_U$, we have the following adjunction (see \cite{Mac})
\begin{equation}\label{monadj}
	EM_U(UM,N)\cong \mathcal C(M,N)
\end{equation}
We now recall the following result from \cite{ABKR}.
\begin{Thm}(see \cite[\S 3]{ABKR})\label{T2.1}
 Let $(U,\theta,\eta)$ be an exact monad on $\mathcal C$ that preserves colimits. Then, $EM_U$ is a Grothendieck category. Further, if $G$ is a projective generator for $\mathcal C$, then $UG$ is a projective generator for $EM_U.$
\end{Thm}

\smallskip

\subsection{Hereditary torsion theory and Gabriel filters on a Grothendieck category}

		Let $\mathcal D$ be any abelian category. Then, we recall from \cite[$\S$1.1]{BR} that a torsion theory on $\mathcal D$ is a pair $(\mathcal T, \mathcal F)$ of strict and full subcategories such that:
	
\begin{enumerate}
	\item $\mathcal D(X, Y) = 0$ for any $X \in \mathcal T$, $Y \in \mathcal F$.
	\item For any $Z \in \mathcal D,$ there exist a short exact sequence
	\begin{equation*}
		0\longrightarrow Z^{\mathcal T}\longrightarrow Z\longrightarrow Z^{\mathcal F}\longrightarrow 0
	\end{equation*}
	where $Z^{\mathcal T}\in \mathcal T$ and $Z^{\mathcal F}\in \mathcal F$.
\end{enumerate}
Further, $(\mathcal T, \mathcal F)$ is said to be hereditary torsion theory if the torsion class $\mathcal T$ is closed under subobjects. 

\smallskip
Since $\mathcal C$ is a Grothendieck category, we know from \cite{BS} that there exists a bijective correspondence between hereditary torsion theories and idempotent kernel functors in $\mathcal C$ given as follows:

\smallskip
For any hereditary torsion theory $\tau = (\mathcal T, \mathcal F)$ on $\mathcal C$, the idempotent kernel functor $\sigma_{\tau}:\mathcal C\longrightarrow \mathcal C$ corresponding to $\tau$ is defined as
\begin{equation}\label{eq2.5}
	\sigma_{\tau}(M) = \sum_{\substack{N\subseteq M \\ N\in \mathcal T}}N
	\end{equation} 
	for any $M\in\mathcal C$. Conversely, for any idempotent kernel functor $\sigma:\mathcal C\longrightarrow \mathcal C,$ there exists an hereditary torsion theory on $\mathcal C$, which we denote by $\tau^{\sigma} = (\mathcal T^{\sigma}, \mathcal F^{\sigma})$, and is given by
	\begin{equation}\label{eq2.6}
		\mathcal T^{\sigma} = \{M\in\mathcal C~|~\sigma(M) = M\}\quad\textup{and}\quad\mathcal F^{\sigma} = \{M\in\mathcal C~|~\sigma(M) = 0\}
		\end{equation}
	 
Further, we know from \cite[\S3]{LLV} that if $\sigma:\mathcal C\longrightarrow \mathcal C$ is an idempotent kernel functor, then for each $M\in\mathcal C$, the Gabriel filter of subobjects of $M$ relative to $\sigma$, denoted by $\mathcal L_{M}^{\sigma}$, is defined as
\begin{equation}\label{eq2.7}
	N\in \mathcal L_{M}^{\sigma} \iff M/N\in \mathcal T_{\sigma}
\end{equation} 
\begin{thm}(see \cite[Proposition 3.3]{LLV})\label{P2.2}
	Let $\mathcal C$ be a Grothendieck category and $\mathcal G = \{G_i\}_{i\in \Lambda}$ be a system of finitely generated projective generators for $\mathcal C$. Let $\{\mathcal L_{G_i}\}_{i\in\Lambda}$ be a family of filters of subobjects of $G_i$ that satisfies the following conditions:
	\begin{enumerate}
		\item For each $i\in \Lambda$, $\mathcal L_{G_i}\neq \emptyset$. In particular, $G_i\in \mathcal L_{G_{i}}$ for each $i\in \Lambda$.\label{p2.1}
		\item If $I\subseteq J\subseteq G_i$ and $I\in \mathcal L_{G_i}$, then $J\in \mathcal L_{G_i}$.\label{p2.2}
		\item If $f\in \mathcal C(G_j,G_i)$ and $I\in \mathcal L_{G_i}$, then $f^{-1}(I)\in \mathcal L_{G_j}$.\label{p2.3}
		
		\item Let $I\subseteq J\in\mathcal L_{G_i}$ such that $f^{-1}(I)\in \mathcal L_{G_j}$ for every $f\in\mathcal C(G_j, J)$. Then, $I\in\mathcal L_{G_i}$.\label{p2.4}
	\end{enumerate}
	 Then, there exists an idempotent kernel functor $\sigma:\mathcal C\longrightarrow \mathcal C$ such that $\mathcal L_{G_i}^{\sigma}=\mathcal L_{G_i}$ for each $i\in\Lambda$.
\end{thm}
Using Proposition \ref{P2.2}, one can define a family of Gabriel filters on a Grothendieck category as follows.
\begin{defn}(see \cite[Proposition 3.2 \& Proposition 3.3]{LLV})\label{D2.3}
	Let $\mathcal C$ be a Grothendieck category and $\mathcal G = \{G_i\}_{i\in \Lambda}$ be a system of finitely generated projective generators for $\mathcal C$. Then, a family of Gabriel filters $\mathcal L=\{\mathcal L_{G_i}\}_{i\in\Lambda}$ is a collection that satisfies the properties (\ref{p2.1}) to (\ref{p2.4}) in Proposition \ref{P2.2}.
\end{defn}
By the above reasoning, it is clear that there exists a bijective correspondence between hereditary torsion theories and families of Gabriel filters $\mathcal L_{G_i}$ on $\mathcal C$ (see \cite[Corollary 3.4]{LLV}).
	\section{Differential torsion theory on Eilenberg-Moore category}
	We continue with $\mathcal C$ being a Grothendieck category and let $\mathcal G=\{G_i\}_{i\in \Lambda}$ be a set of finitely generated projective generators for $\mathcal C$. In this section, we shall prove that if $(U, \theta, \eta)$ is a monad on $\mathcal C$ that is exact and preserves colimits, then the hereditary torsion theories on $EM_{U}$ are differential.

	\begin{thm}\label{P3.1}
		Let $(U, \theta, \eta)$ be a monad on $\mathcal C$ that is exact and preserves colimits. Then, $EM_{U}$ is a Grothendieck category and the collection $\{UG_i\}_{i\in \Lambda}$ gives a set of finitely generated projective generators for $EM_U$.
	\end{thm}
	\begin{proof}
From \cite[Proposition 3.1]{ABKR}, it is clear that $EM_{U}$ is an abelian category in which colimits exists and filtered colimits are exact. It now remains to show that $\{UG_{i}\}_{i\in \Lambda}$ is a set of finitely generated projective generators for $EM_U$. Let  $\{ M_t\}_{t\in T}$ be any directed system of objects in $EM_U$. Since each $G_i\in \mathcal C$ is finitely generated, we get $$EM_{U}(UG_i, \underset{t\in T}\varinjlim M_t) \cong \mathcal C(G_i, \underset{t\in T}\varinjlim M_t) = \underset{t\in T}\varinjlim \mathcal C(G_i, M_t) \cong \underset{t\in T}\varinjlim EM_{U}(UG_i,  M_t)$$ Hence each $UG_i$ is a fintely generated in $EM_U$. Further, as each $G_i$ is projective, $\mathcal C(G_i, -) \cong EM_{U}(UG_i, -)$ is exact. Hence, each $UG_i$ is projective in $EM_U$.

\smallskip
We now consider a subobject $M'\subsetneq M$ in $EM_U$ and let $\iota:M'\longrightarrow M$ denote a monomorphism in $EM_U$. Let for all $i\in \Lambda$ and for any morphism $h\in EM_{U}(UG_i, M)$, there exists a morphism $g\in EM_U(UG_i, M')$ such that $\iota\circ g=h$. Then, $\iota\circ g\circ \eta_{G_i} = h\circ \eta_{G_i}$ in $\mathcal C(G_i, M)$. Now, consider any morphism $\phi\in \mathcal C(G_i, M)$. By (\ref{monadj}), we have a corresponding morphism $\hat{\phi}\in EM_U(UG_i, M)$ given by $\hat{\phi}:UG_i\xrightarrow{U\phi}UM\xrightarrow{f_M}M.$ Then, it follows from the above that there exists a morphism $\hat{g}\in EM_U(UG_i, M')$ such that $\iota\circ \hat g = \hat{\phi}$ in $EM_U$. Therefore, $\iota\circ \hat g\circ \eta_{G_i} = \hat{\phi}\circ \eta_{G_i}$ in $\mathcal C$. Since the composition $\hat{\phi}\circ \eta_{G_i}:G_i\longrightarrow M$ gives back the morphism $\phi:G_i\longrightarrow M$ in $\mathcal C$, and $\{G_i\}_{i\in \Lambda}$ is a family of generators for $\mathcal C$, we get a contradiction from \cite[\S1.9]{Gro}. Therefore, $\{UG_i~|~i\in \Lambda\}$ forms a set of generators for $EM_{U}$. 
	\end{proof}
	\begin{defn}\label{D3.2}
		Let $(U, \theta, \eta)$ be a monad on $\mathcal C$. Then, we say that a natural transformation $\delta:U\longrightarrow U$ is a derivation on $U$ if the following diagram commutes.
		\begin{equation}\label{eq3.1}
			\begin{CD}
				UU @>1\ast \delta +\delta\ast 1 >> UU \\
				@VV \theta V @VV \theta V \\
				U @>\delta>> U \\
			\end{CD}
		\end{equation} 
	\end{defn}
We now consider a monad $U$ on $\mathcal C$ that is exact and preserves colimits. Then, by Proposition \ref{P3.1}, $EM_U$ is a Grothendieck category. Let $\{UG_i\}_{i\in \Lambda}$ be a set of finitely generated projective generators for $EM_U$. Then, by Definition \ref{D2.3}, a family $\mathcal L = \{\mathcal L_{UG_{i}}\}_{i\in \Lambda}$ of Gabriel filters on $EM_U$ is a collection that satisfies the following conditions:
	\begin{enumerate}
		\item For each $i\in \Lambda$, $\mathcal L_{UG_{i}}\neq \emptyset$. In particular, $UG_i\in \mathcal L_{UG_{i}}$ for each $i\in \Lambda$.\label{P1}
		
		\item If $I\subseteq J\subseteq UG_i$ and $I\in \mathcal L_{UG_{i}}$, then $J\in \mathcal L_{UG_{i}}$.\label{P2}
		
		\item Let $f\in EM_U(UG_j,UG_i)$ and $I\in \mathcal L_{UG_{i}}$. Then, $f^{-1}(I)\in \mathcal L_{UG_{j}}$.\label{P3}
		
		\item Let $I\subseteq J\in\mathcal L_{UG_i}$ such that $f^{-1}(I)\in \mathcal L_{UG_{j}}$ for every $f\in EM_U(UG_j, J)$. Then $I\in\mathcal L_{UG_{i}}$. \label{P4}
	\end{enumerate}

\begin{defn}\label{D3.3}
	Let $(U, \theta, \eta)$ be a monad on $\mathcal C$ that is exact and preserves colimits and let $\delta$ be a derivation on $U$. Let $\mathcal L = \{\mathcal L_{UG_{i}}\}_{i\in \Lambda}$ be a family of Gabriel filters on $EM_{U}$. Then, we say that the family $\mathcal L=\{\mathcal L_{UG_{i}}\}_{i\in \Lambda}$ on $EM_U$  is $\delta$-invariant if for each $i\in\Lambda$ and $I\in \mathcal L_{UG_{i}}$, there exists some $J\in\mathcal L_{UG_{i}}$ such that $\delta_{G_i} (J)\subseteq I.$ 
	\end{defn}
	\textbf{Note:} For any morphism $f\in\mathcal C(M,N)$ and any subobject $L\subseteq M$ in $\mathcal C$, the restriction map of $f$ acting on $L$ is denoted as $f(L)$.
	\begin{Thm}\label{T3.4}
		Let $(U, \theta, \eta)$ be a monad on $\mathcal C$ that is exact and preserves colimits and let $\delta$ be a derivation on $U$. Then, every family $\mathcal L = \{\mathcal L_{UG_{i}}\}_{i\in \Lambda}$ of  Gabriel filters on $EM_{U}$ is $\delta$-invariant.
	\end{Thm}
	\begin{proof}
		Let $I\in\mathcal L_{UG_{i}}$ for some $i\in\Lambda$. We set
		\begin{equation}\label{eq3.2}
			J := \underset{P\subseteq |J|}\sum P,~\text{where}~|J| = \{P\subseteq I~|~ \delta_{G_i}(P)\subseteq I~in ~EM_U\}
		\end{equation}
			  From (\ref{eq3.2}), it is clear that $\delta_{G_i}(J)\subseteq I\subseteq UG_i$ in $EM_U$. It now remains to show that $J\in\mathcal L_{UG_{i}}$. We consider a morphism $f\in EM_U(UG_j,UG_i)$ such that $Im(f)\subseteq I$. From property (\ref{P4}), it is enough to show that $f^{-1}(J)\in \mathcal L_{UG_{j}}$. For this, we shall first show that $(\delta_{G_i}\circ f)^{-1}(I)\subseteq f^{-1}(J)$. Let $N\subseteq (\delta_{G_i}\circ f)^{-1}(I)$ be any arbitrary subobject in $EM_U$. Then, $(\delta_{G_i}\circ f)(N) = \delta_{G_i}
			  (f(N))\subseteq I$. From (\ref{eq3.2}), we get that $f(N)\subseteq J$. Therefore, $N\subseteq f^{-1}(J)$. Hence $(\delta_{G_i}\circ f)^{-1}(I)\subseteq f^{-1}(J)$. Further, from property (\ref{P3}), we get $(\delta_{G_i}\circ f)^{-1}(I)\in \mathcal L_{UG_{j}}$. The result now follows from property (\ref{P2}).
	\end{proof}
Let $\mathcal L = \{\mathcal L_{UG_{i}}\}_{i\in \Lambda}$ be a family of Gabriel filters on $EM_{U}$. Then, by \cite[Theorem 2.1]{GG}, the corresponding hereditary torsion class $\mathcal T$ is defined as follows: 
	\begin{equation}\label{eq3.3}
		Ob(\mathcal T):=\{M\in EM_{U}~|~Ker(f:UG_i\longrightarrow M)\in \mathcal L_{UG_{i}},~~\text{for all}~~ f\in EM_{U}(UG_i, M)~~\text{and}~~\text{for all}~~i\in \Lambda\}
	\end{equation}
	Further, for any hereditary torsion theory $\tau = (\mathcal T, \mathcal F)$ and any object $M\in EM_{U}$, the $\tau$-torsion subobject of $M$ is given by
	\begin{equation}\label{eq3.4}
		M_{\tau} := \underset{N\subseteq |M_{\tau}|}\sum N
	\end{equation}
	where  $|M_{\tau}| = \{N\subseteq M~|~ Ker(f:UG_i\longrightarrow N) \in \mathcal L_{UG_{i}},~ \textup{for all} ~f\in EM_{U}(UG_i,N)~ \text{and for all} ~i\in \Lambda\}$. 
	
	\smallskip
	Next, for a given monad $U$ that is equipped with a derivation $\delta$, we define the $\delta$-derivation on any object $M$ in $EM_{U}$.
	\begin{defn}\label{D3.5}
 Let $(U, \theta, \eta)$ be a monad on $\mathcal C$ that is equipped with a derivation $\delta$ and let $M$ be an object in $EM_{U}$. Then, we say that a morphism $D:M\longrightarrow M$ in $\mathcal C$ is a $\delta$-derivation on $M$ if the following diagram commutes.	
\begin{equation}\label{eq3.5}
			\begin{CD}
				UM @>UD+\delta_{M} >> UM \\
				 @VV f_{M}V @VV f_{M}V \\
				M @>D>> M \\
			\end{CD}
		\end{equation} 
	\end{defn}
\begin{defn}\label{D3.6}
Let $\tau = (\mathcal T, \mathcal F)$ be a hereditary torsion theory on $EM_{U}$. Then, we say that $\tau = (\mathcal T, \mathcal F)$ is differential if for any object $M\in EM_{U}$ and any $\delta$-derivation $D$ on $M$, $D(M_{\tau})\subseteq M_{\tau}$.			\end{defn}
\begin{Thm}\label{T3.7}
Let $U$ be a monad on $\mathcal C$ that is exact and preserves colimits and let $\delta$ be a derivation on $U$. Then, every hereditary torsion theory on $EM_{U}$ is differential.		
\end{Thm}
\begin{proof}
Let $\tau=(\mathcal T, \mathcal F)$ be a hereditary torsion theory on $EM_U$ and let $\mathcal L = \{\mathcal L_{UG_{i}}\}_{G_i\in \mathcal C}$ be a family of Gabriel filters corresponding to $\tau$. Consider an object $M\in EM_{U}$ and let $D:M\longrightarrow M$ be a $\delta$-derivation on $M$. To prove that $D(M_{\tau})\subseteq M_{\tau}$, it is enough to show that for any $N\overset{i_1}\hookrightarrow M$ such that $N\in |M_{\tau}|$, $DN \in |M_{\tau}|$. We start with a morphism  $g\in EM_{U}(UG_i, DN)$. Then, by (\ref{monadj}), there exists a corresponding morphism  $\hat{g}\in \mathcal C(G_i, DN)$ given by $\hat{g}:G_i\xrightarrow{\eta_{G_i}} UG_i\xrightarrow{g}DN$. Let $\tilde{D}:N\longrightarrow DN$ denote an epimorphism in $\mathcal C$ and $i_2:DN\hookrightarrow M$ denote a monomorphism in $EM_U$. Then, $i_2\circ \tilde{D} = D\circ i_1$. Since $G_i$ is projective, there exists a morphism $\hat{h}\in \mathcal C(G_i, N)$ such that $\tilde{D}\circ \hat{h} = \hat{g} = g\circ \eta_{G_i}.$ Let $h:UG_i\xrightarrow{U\hat{h}} UN\xrightarrow{f_{N}} N$ be the corresponding morphism in $EM_{U}$. Set $K = Ker(h)$. Then, $K\in \mathcal L_{UG_{i}}$. We know from Theorem \ref{T3.4} that every family of Gabriel filters is $\delta$-invariant. Hence, there exists some $J \in \mathcal{L}_{UG_{i}}$ such that $\delta_{G_i}(J) \subseteq K$. Set $I = K\cap J$. 

\smallskip
We now consider the following commutative diagram in $EM_{U}$.
\begin{equation*}
	\begin{CD}
		UG_i @>Ui_1\circ U\hat{h} >> UM @>UD+ \delta_{M} >> UM \\
		@. @VV f_{M}V @VV f_{M}V \\
		@.M @>D>> M \\
	\end{CD}
\end{equation*} 
 Then,
\begin{equation}\label{eq3.6}
	D(f_{M}(Ui_1(U\hat{h}(I))))= f_{M}(UD(Ui_1(U\hat{h}(I)))) + f_{M}( \delta_{M}(Ui_1(U\hat{h}(I))))
\end{equation} 
Since, $f_M\circ Ui_1=i_1\circ f_N$ and $f_N\circ U\hat{h}=h$, $$D(f_{M}(Ui_1(U\hat{h}(I))))=D(i_1(f_{N}( U\hat{h}(I))))=D( i_1(h(I))) = 0$$ Further, as $\delta_N\circ U\hat{h} = U\hat{h}\circ \delta_{G_i}$, we obtain
\begin{align*} 
	f_{M}(\delta_{M}( Ui_1( U\hat{h}(I))))&=f_{M}(Ui_1( \delta_{N}( U\hat{h}(I))))\\\notag &=f_{M}(Ui_1(U\hat{h}( \delta_{G_i}(I))))\\\notag & =i_1(f_{N}( U\hat{h}( \delta_{G_i}(I))))\\\notag &= i_1( h(\delta_{G_i}(I)))
\end{align*}
 Since $\delta_{G_i}(I)\subseteq I\subseteq K$, $f_{M}(\delta_{M}( Ui_1( U\hat{h}(I))))= i_1(h(\delta_{G_i}(I))) = 0.$ This gives $f_{M}(UD( Ui_1(U\hat{h}(I)))) = 0$. Further, as $i_2\circ \tilde{D} = D\circ i_1$ and $\tilde{D}\circ \hat{h} = \hat{g}$, we get
\begin{align*}
	0=f_{M}(UD( Ui_1(U\hat{h}(I))))&=f_{M}(Ui_2( U\tilde{D}( U\hat{h}(I)))) \\\notag&= f_{M}(Ui_2( U\hat{g}(I)))\\\notag&=f_{M}(Ui_2(Ug (U\eta_{G_i}(I))))\\\notag &= i_2(g( \theta_{G_i}( U\eta_{G_i}(I))))
\end{align*}  
Since $\theta_{G_i} \circ U\eta_{G_i} = 1_{UG_i}$, $f_{M}(UD( Ui_1(U\hat{h}(I)))) = i_2( g(I))=0.$ As $i_2$ is a monomorphism, it follows that $g(I)=0$. Finally, by property (\ref{P2}), $Ker(g)\in \mathcal L_{UG_{i}}$. Hence $DN\in |M_{\tau}|$. 
\end{proof}

\section{On module of quotients}
Let $(U, \theta, \eta)$ be a monad on $\mathcal C$ that is exact and preserves colimits and let $\tau = (\mathcal T, \mathcal F)$ be a hereditary torsion theory on $EM_{U}$. Let $\mathcal L = \{\mathcal L_{UG_{i}}\}_{G_i\in \mathcal C}$ be the family of Gabriel filters corresponding to $\tau$. Consider an object $M$ in $EM_{U}$ and let $M_{\tau}$ denote the torsion submodule of $M$. We know from \cite[\S 2.2]{GG} that there exists a functor $H_{\tau} : EM_{U}\longrightarrow EM_{U}$, $M\mapsto H_{\tau}(M)$  defined as
\begin{equation}\label{eq4.1}
	H_{\tau}(M)(UG_i) = \underset{I\in \mathcal L_{UG_{i}}} \varinjlim EM_{U}(I, M/M_{\tau})
	\end{equation} 
 $H_{\tau}(M)$ is called as the module of quotients of $M$. It is clear from \cite[\S 2.2]{GG} that $H_{\tau}$ is a left exact functor and for any $\tau$-torsion object $M\in EM_U$, $H_{\tau}(M) = 0$.

\smallskip
Let $\zeta_i:EM_{U}(UG_i, M/M_{\tau})\longrightarrow \varinjlim EM_{U}(I, M/M_{\tau})$ be the canonical morphism. Then, for any morphism $f\in EM_{U}(UG_i, M),$ there exists a morphism $\Phi_{M}:M\longrightarrow H_{\tau}(M)$ in $EM_{U}$ given by (see \cite[\S 2.2]{GG})
\begin{equation}\label{eq4.2}
	\Phi_{M}(UG_i)(f)=\zeta_{i}(p\circ f)
\end{equation}
 where $p:M\longrightarrow  M/M_{\tau}$ denote the canonical epimorphism in $EM_{U}$. 
Further, it is clear from the construction of the map $\Phi_{M}$ that for any morphism $h:UG_j\longrightarrow UG_i$ in $EM_{U}$, there exists a morphism $H_{\tau}(M)(h):H_{\tau}(M)(UG_i)\longrightarrow H_{\tau}(M)(UG_j)$ such that the following diagram commutes.
 	\begin{equation}\label{eq4.3}
 	\begin{CD}
 		EM_{U}(UG_i, M) @>\Phi_{M}(UG_i)(h) >> H_{\tau}(M)(UG_i) \\
 		@VV EM_{U}(-, M)(h)V @VV H_{\tau}(M)(h) V \\
 		EM_{U}(UG_j, M) @>\Phi_{M}(UG_i)(h)>> H_{\tau}(M)(UG_j)\\
 	\end{CD}
 \end{equation} 
 Therefore, $\Phi_{M}$ is functorial in $M$. Furthermore, both $Ker(\Phi_{M})$ and $Coker(\Phi_{M})$ are $\tau$-torsion modules for any object $M\in EM_U$ (see \cite[Proposition 2.4]{GG}).
 
 \smallskip
 In a manner similar to (see \cite[Lemma 2.10]{AB}), we now define the $\delta$-derivation on the module of quotients $H_{\tau}(M)$ for a torsion free module $M\in EM_U$.
 \begin{lem}\label{L4.1}
 	Let $M$ be a torsion free module in $EM_{U}$ equipped with a $\delta$-derivation $D$ and let $f\in H_{\tau}(M)(UG_i)$. Then, $f$ represents a morphism $f:I\longrightarrow M$ for some $I\in \mathcal L_{UG_{i}}$. Let $J\in \mathcal L_{UG_{i}}$ be such that $\delta_{G_i}(J)\subseteq I$. Set $K=I\cap J$. Then, the map $\bar{D}f:K\longrightarrow M$ defined by setting $$\bar{D}f(K)= D(f(K)) - f(\delta_{G_i}(K))$$ is a morphism in $H_{\tau}(M)(UG_i)$.
 \end{lem}
 
 \begin{proof}
 	Clearly, $\bar{D}f(K)\subseteq M$. Since $K\in \mathcal L_{UG_i}$, there exists an object $K'\in \mathcal L_{UG_i}$ such that $\delta_{G_i}(K')\subseteq K$. To show that $\bar{D}f$ is a morphism in $EM_U$, it is enough to show that the following diagram commutes.
 	\begin{equation*}
 		\begin{CD}
 		UK'@>Ui	>>UK @>U(\bar{D}f) >> UM \\
 		@VV f_{K'}V	@VV f_{k}V @VV f_{M}V \\
 		K'@> i >>K @>\bar{D}f>> M \\
 		\end{CD}
 	\end{equation*} 
 	It is already clear that the first square commutes in $EM_U$.
 	Further, we note that 
 	\begin{align}\label{eq4.4}
 		\bar{D}f(i(f_{K'}(UK')))= D(f(i(f_{K'}(UK'))))-f(\delta_{G_i}(i(f_{K'}(UK'))))
 	\end{align}
 	On the other hand we have
 	\begin{align}\label{eq4.5}
 		f_{M}(U(\bar{D}f)(Ui(UK'))) &= f_{M}(U(\bar{D}f(i(K'))))\\\notag&= f_{M}(U(D(f(i(K'))-f(\delta_{G_i}(i(K'))))) \\\notag&= f_{M}( UD(Uf(Ui(UK'))))-f_{M}(Uf(U\delta_{G_i}(Ui(UK'))))
 	\end{align}
 	Since $K\subseteq I$ and $f:I\longrightarrow M$, we have the following commutative diagram in $EM_U$, where we continue to denote the restriction map $f|_{K}:K\longrightarrow M$ by $f$.
 	\begin{equation*}
 		\begin{CD}
 		UK'@>Ui>>UK @>Uf >> UM @>UD+\delta_{M} >> UM \\
 		@VV f_{K'}V	@VV f_{K}V @VV f_{M}V @VV f_{M}V \\
 		K'@> i>>K@>f>>M @>D>> M \\
 		\end{CD}
 	\end{equation*} 
 	Then $f_{M}(UD(Uf(Ui(UK')))) + f_{M}(\delta_{M}(Uf(Ui(UK'))))  = D(f(i(f_{K'}(UK'))).$ Accordingly, we get
 	\begin{equation}\label{eq4.6}
 		f_{M}(U(\bar{D}f)(Ui(UK'))) =  D(f(i(f_{K'}(UK')))) - f_{M}(\delta_{M}(Uf(Ui(UK')))) - f_{M}(Uf(U\delta_{G_i}(Ui(UK'))))
 	\end{equation}
 	Therefore, from (\ref{eq4.4}) and (\ref{eq4.6}), it is sufficient to show that \begin{align}\label{eq4.7}
 		f_{M}(\delta_{M}(Uf(Ui(UK')))) +f_{M}(Uf(U\delta_{G_i}(Ui(UK')))) = f(\delta_{G_i}(i(f_{K'}(UK'))))
 	\end{align}
 As $f_K$ and $\delta_K$ are restrictions of $\theta_{G_i}$ and $\delta_{UG_i}$, we obtain
 	 \begin{align*}
 		\delta_{G_i}(i(f_{K'}(UK')))&=\delta_{G_i}(f_{K}(Ui(UK')))\\\notag &=\delta_{G_i}(\theta_{G_i}(Ui(UK'))) \notag\\&= \theta_{G_i}((1 \ast \delta_{G_i}+\delta_{G_i}\ast 1)(Ui(UK')))\notag\\& = \theta_{G_i}( (U\delta_{G_i}+\delta_{UG_i})(Ui(UK')))\notag\\&=f_K((U\delta_{G_i}+\delta_{UG_i})(Ui(UK')))
 	\end{align*}
 	where the last equality follows from the fact that $\delta_{G_i}(K')\subseteq K$.  
 	Further, $\delta_M\circ Uf = Uf\circ \delta_K$ and $f_M\circ Uf = f\circ f_K$ gives
 	\begin{align*}
 			f_{M}(\delta_{M}(Uf(Ui(UK')))) +f_{M}(Uf(U\delta_{G_i}(Ui(UK')))) &=f_{M}(Uf(\delta_{K}(Ui(UK'))))+f_{M}(Uf( U\delta_{G_i}(Ui(UK'))))\notag\\&=f_{M}(Uf(\delta_{UG_i}(Ui(UK'))))+f_{M}(Uf( U\delta_{G_i}(Ui(UK'))))\notag\\
 		&=f(f_{K}(\delta_{UG_i}(Ui(UK'))))+ f(f_{K}( U\delta_{G_i}(Ui(UK'))))\\\notag&=f(\delta_{G_i}(i(f_{K'}(UK'))))
 	\end{align*}
 \end{proof}
By Lemma \ref{L4.1}, there exists a morphism $\bar{D}:H_{\tau}(M)\longrightarrow H_{\tau}(M)$ in $\mathcal C$ defined as 
\begin{equation}\label{eq4.8}
	\bar{D}(UG_i):H_{\tau}(M)(UG_i)\longrightarrow H_{\tau}(M)(UG_i), \quad f\mapsto \bar{D}f
\end{equation}
	for $M\in EM_U$ and $f\in H_{\tau}(M)(UG_i)$. 
\begin{lem}\label{L4.2}
Let $M$ be a torsion free module in $EM_{U}$ equipped with a $\delta$-derivation $D$. Let $\bar{D}:H_{\tau}(M)\longrightarrow H_{\tau}(M)$ be a morphism as defined in (\ref{eq4.8}). Then, the following diagram commutes.
	\begin{equation}\label{eq4.9}
	\begin{CD}
	 M @>\Phi_{M} >> H_{\tau}(M) \\
		@VV D V @VV \bar{D} V \\
		M @>\Phi_{M}>> H_{\tau}(M)\\
	\end{CD}
\end{equation} 
\end{lem}
\begin{proof}
	Let $f:UG_i\longrightarrow M$ be a morphism in $EM_U$. By (\ref{monadj}), we have a corresponding morphism $\hat{f}\in \mathcal C(G_i, M)$ given by $\hat{f}:G_i\xrightarrow{\eta_{G_i}} UG_i\xrightarrow{f} M$. Then, $D\circ \hat{f}\in \mathcal C(G_i, M)$, which by (\ref{monadj}) corresponds to a morphism in $EM_U(UG_i, M)$ given by $UG_i\xrightarrow{U\eta_{G_i}} UUG_i\xrightarrow{Uf}UM\xrightarrow{UD} UM\xrightarrow{f_{M}} M$. 
	Denote $\hat{D}(f) := f_{M}\circ UD\circ Uf\circ U\eta_{G_i}$.
Observe that to establish the commutativity of the diagram (\ref{eq4.9}), it suffices to show that the following diagram commutes.
\begin{equation}\label{eq4.10}
		\begin{CD}
		EM_{U}(UG_i, M) @>\Phi_{M}(UG_i) >> H_{\tau}(M)(UG_i) \\
		@VV \hat{D}(UG_i) V @VV \bar{D}(UG_i) V \\
		EM_{U}(UG_i, M) @>\Phi_{M}(UG_i)>> H_{\tau}(M)(UG_i)\\
	\end{CD}
	\end{equation} 
		Since $\theta_{G_i}\circ U\eta_{G_i} = 1_{UG_i}$, we have
		\begin{align*}
			\bar{D}(f)(UG_i)&= D(f(UG_i))-f(\delta_{G_i}(UG_i))\\\notag
			&= D(f(1_{UG_i}(UG_i)))-f(\delta_{G_i}(UG_i))\\\notag
				&=D(f(\theta_{G_i}(U\eta_{G_i}(UG_i))))-f(\delta_{G_i}(UG_i))\notag
			\end{align*}
			 Further, from the equalities $f\circ \theta_{G_i} = f_M\circ Uf$,  
			$\delta_{M}\circ Uf=Uf\circ \delta_{UG_i}$ and $\delta_{UG_i}\circ U\eta_{G_i} = U\eta_{G_i}\circ \delta_{G_i}$ we get 
			\begin{align*}
			\bar{D}(f)(UG_i)&=D(f_M(Uf(U\eta_{G_i}(UG_i))))-f(\delta_{G_i}(UG_i))\\\notag
				&=f_{M}(UD(Uf(U\eta_{G_i}(UG_i))))+f_{M}(\delta_{M}(Uf( U\eta_{G_i}(UG_i))))-f(\delta_{G_i}(UG_i))\\\notag
				&=f_{M}(UD(Uf(U\eta_{G_i}(UG_i))))+f_{M}(Uf(\delta_{UG_i}(U\eta_{G_i}(UG_i))))-f(\delta_{G_i}(UG_i))\\\notag
				&=f_{M}(UD(Uf(U\eta_{G_i}(UG_i))))+f_{M}(Uf(U\eta_{G_i}(\delta_{G_i}(UG_i))))-f(\delta_{G_i}(UG_i))	
				\end{align*}
Finally, as the composition $UG_i\xrightarrow{U\eta_{G_i}} UUG_i\xrightarrow{Uf} UM\xrightarrow{f_M} M$ gives back the morphism $f:UG_i\longrightarrow M$, we get $\bar{D}(f)(UG_i)=\hat{D}(f)(UG_i).$ Hence the diagram (\ref{eq4.10}) commutes. 
\end{proof}
\begin{Thm}\label{T4.3}
 Let $(U, \theta, \eta)$ be a monad on $\mathcal C$ that is exact and preserves colimits and let $\delta$ be a derivation on $U$. Let $\tau = (\mathcal T, \mathcal F)$ be a hereditary torsion theory on $EM_{U}$ and $M$ be a torsion free module in $EM_{U}$ equipped with a $\delta$-derivation $D$. Then, $\bar{D}:H_{\tau}(M)\longrightarrow  H_{\tau}(M)$ is a $\delta$-derivation on $H_{\tau}(M)$.		
\end{Thm}
	\begin{proof}
We begin with a morphism $f\in H_{\tau}(M)(UG_i)$. Then, $f$ represents a morphism $f:I\longrightarrow M$ for some $I\in \mathcal L_{UG_{i}}$. 
Consider the following diagram in $\mathcal C$.
\begin{equation}\label{eq4.11}
	\begin{CD}
		UI @> Uf >> UM @>U\Phi_{M} >> UH_{\tau}(M) @>U\bar{D}+\delta_{H_{\tau}(M)} >> UH_{\tau}(M) \\
		@VV f_{I} V@VV f_{M}V @VV f_{H_{\tau}(M)}V @VV f_{H_{\tau}(M)}V \\
		I@>f>>M@>\Phi_{M}>>H_{\tau}(M) @>\bar{D}>> H_{\tau}(M)\\
	\end{CD}
\end{equation} 
Note that the first two squares in diagram (\ref{eq4.11}) commutes. To show that $\bar{D}$ is a $\delta$-derivation on $H_{\tau}(M)$, it is sufficient to show that the following equality holds
\begin{equation*}
	f_{H_{\tau}(M)}\circ U\bar{D}\circ U\Phi_{M}\circ Uf + f_{H_{\tau}(M)}\circ \delta_{HM}\circ U\Phi_{M}\circ Uf= \bar{D}\circ \Phi_{M}\circ f\circ f_{I}
\end{equation*}
 Since $\delta_{H_{\tau}(M)}\circ U\Phi_{M}=U\Phi_{M}\circ \delta_{M}$, $f_{H_{\tau}(M)}\circ U\Phi_{M} = \Phi_{M}\circ f_M$ and from Lemma \ref{L4.2}, $\bar{D}\circ \Phi_{M}=\Phi_M\circ D$, we obtain 
\begin{align*}
	f_{H_{\tau}(M)}\circ U\bar{D}\circ U\Phi_{M}\circ Uf + f_{H_{\tau}(M)}\circ \delta_{HM}\circ U\Phi_{M}\circ Uf &=f_{H_{\tau}(M)}\circ U\Phi_{M}\circ UD\circ Uf+ f_{H_{\tau}(M)}\circ U\Phi_{M}\circ \delta_{M}\circ Uf\\\notag &=\Phi_{M}\circ f_{M}\circ UD\circ Uf+ \Phi_{M}\circ f_{M}\circ \delta_{M}\circ Uf\\\notag &= \Phi_{M}\circ f_{M}\circ (UD+\delta_{M})\circ Uf\\\notag &= \Phi_{M}\circ D\circ f_{M}\circ Uf\\\notag&=\bar{D}\circ \Phi_{M}\circ f_{M}\circ Uf
\end{align*}
From the commutativity of the first square in diagram (\ref{eq4.11}), we have $f_{M}\circ Uf=f\circ f_{I}$. Hence the result follows.
\end{proof}
\begin{Thm}\label{T4.4}
	The derivation $\bar{D}$ obtained in Theorem \ref{T4.3} is the unique $\delta$-derivation on $H_{\tau}(M)$ that lifts the $\delta$-derivation $D$ on $M$.
\end{Thm}
\begin{proof}
Since $\bar{D}\circ \Phi_{M} = \Phi_{M}\circ D$, $\bar{D}$ lifts the $\delta$-derivation $D$ on $M$. Let $\bar{D}'$ be another $\delta$-derivation on $H_{\tau}(M)$ that lifts the $\delta$-derivation $D$ on $M$. Then,
		\begin{align*}
			(\bar{D}-\bar{D}')\circ f_{H_{\tau}(M)} = f_{H_{\tau}(M)}\circ (U\bar{D}+\delta_{H_{\tau}(M)})-f_{H_{\tau}(M)}\circ (U\bar{D}'+\delta_{H_{\tau}(M)}) = f_{H_{\tau}(M)}\circ (U\bar{D}-U\bar{D}')
				\end{align*}
			This shows that $\bar{D}-\bar{D}'$ is a morphism on $H_{\tau}(M)$ in $EM_{U}$. Further, by Lemma \ref{L4.2}, we have $(\bar{D}-\bar{D}')\circ \Phi_{M} = \Phi_{M}\circ (D-D) = 0$. Hence, there is an induced morphism  $Coker(\Phi_{M})\longrightarrow H_{\tau}(M)$ such that $(\bar{D}-\bar{D}'):H_{\tau}(M)\longrightarrow H_{\tau}(M)$ factors through it. Now, we know from \cite[Proposition 2.4 $\&$ Theorem 2.5]{GG}, that $Coker(\Phi_{M})\in \mathcal T$ and $H_{\tau}(M)\in \mathcal F$. Therefore, the morphism $Coker(\Phi_{M})\longrightarrow H_{\tau}(M)$ is zero, and hence $\bar{D} = \bar{D}'$.
	\end{proof}
	\begin{Thm}\label{T4.5}
		Let $U$ be a monad on $\mathcal C$ that is exact and preserves colimits and let $\delta$ be a derivation on $U$. Let $\tau = (\mathcal T, \mathcal F)$ be a hereditary torsion theory on $EM_{U}$ and $M$ be a module in $EM_{U}$ equipped with a $\delta$-derivation $D$. Then, there exists a unique $\delta$-derivation $\bar{D}$ on $H_{\tau}(M)$ that lifts the $\delta$-derivation $D$ on $M$.
	\end{Thm}
	\begin{proof}
		We observed in Theorem \ref{T3.7} that the torsion submodule $M_{\tau}$ of $M$ is $D$-differential. Therefore, $D$ extends to a $\delta$-derivation $D$ on $M/M_{\tau}$. Since $M/M_{\tau}$ is $\tau$-torsion free, by Theorem \ref{T4.3} there exists a $\delta$-derivation $\bar{D}:H_{\tau}(M/M_{\tau}) = H_{\tau}(M)\longrightarrow H_{\tau}(M)=H_{\tau}(M/M_{\tau})$ that uniquely lifts the $\delta$-derivation $D$ on $M/M_{\tau}$, which clearly gives a lift of derivation $D$ on $M$. Further, the following diagram commutes.
			\begin{equation}\label{eq4.12}
			\begin{CD}
				M @>p >> M/M_{\tau} @>\Phi_{M/M_{\tau}} >> H_{\tau}(M/M_{\tau}) = H_{\tau}(M) \\
				@VV DV @VV DV @VV \bar{D}V \\
				M@>p>>M/M_{\tau} @>\Phi_{M/M_{\tau}}>> H_{\tau}(M/M_{\tau}) = H_{\tau}(M)\\
			\end{CD}
		\end{equation} 
		We now consider another $\delta$ derivation $\bar{D}'$ on $H_{\tau}(M)$ that also lifts the $\delta$-derivation $D$ on $M$. Then, we have the following commutative diagram.
		\begin{equation}\label{eq4.13}
			\begin{CD}
				M @> \Phi_{M} >> H_{\tau}(M)= H_{\tau}(M/M_{\tau})\\
				@VV D V@VV \bar{D}'V \\
				M @> \Phi_{M} >> H_{\tau}(M)= H_{\tau}(M/M_{\tau})\\
			\end{CD}
		\end{equation} 
	  Then, $\bar{D}'\circ \Phi_{M/M_{\tau}}\circ p =\bar{D}'\circ \Phi_{M} =\Phi_{M}\circ D = \Phi_{M/M_{\tau}}\circ p\circ D = \bar{D}\circ \Phi_{M/M_{\tau}}\circ p$. Since $p:M\longrightarrow M/M_{\tau}$ is an epimorphism in $EM_{U}$, $$\bar{D}'\circ \Phi_{M/M_{\tau}}=\bar{D}\circ \Phi_{M/M_{\tau}} = \Phi_{M/M_{\tau}}\circ D$$ This shows that $\bar{D}'$ is a lift of $D$ on $M/M_{\tau}$. The result now follows from Theorem \ref{T4.4}. 
	\end{proof}
\begin{center}
	\textbf{Declarations}
\end{center}
\smallskip
\noindent\textbf{Funding} Divya Ahuja is financially supported by the Council of Scientific $\&$ Industrial Research (CSIR), India [09/086(1430)/2019-EMR-I].

\smallskip
\noindent\textbf{Competing Interests} The authors declare no competing interests related to this research.

\end{document}